\documentclass{article}[12pt]

\usepackage{lineno,hyperref}
\modulolinenumbers[5]

\usepackage{amssymb,amsfonts,amsmath}
\usepackage[all,arc]{xy}
\usepackage{fullpage}
\usepackage{xspace}
\usepackage{amssymb, mathrsfs}
\usepackage{amsmath, graphicx, amsthm}
\usepackage{times}
\usepackage{pgf}
\usepackage{tikz}
\usepackage{tikz-cd}
\usetikzlibrary{matrix,arrows,decorations.pathmorphing}
\usepackage{enumerate}
\usepackage{hyperref}
\hypersetup{
    colorlinks=true,
    linkcolor=blue,
    filecolor=magenta,      
    urlcolor=cyan,
    citecolor=red,
    bookmarks=true
   }
\usepackage{cleveref}
\usepackage{mathtools}
\usepackage{latexsym}
\usepackage{textcomp}
\usepackage{amscd}
\usepackage{amsxtra}
\usepackage{stmaryrd}
\usepackage{mathrsfs}
\usepackage{todonotes}

\newtheorem{thm}{Theorem}[section]
\newtheorem{cor}[thm]{Corollary}
\newtheorem{pr}[thm]{Proposition}
\newtheorem{lm}[thm]{Lemma}

\theoremstyle{definition}
\newtheorem{defn}[thm]{Definition}

\newtheorem{notn}[thm]{Notation}

\theoremstyle{remark}
\newtheorem{rem}[thm]{Remark}

\newcommand{\tit}{\textit}

\newcommand{\mfrak}{\mathfrak}

\newcommand{\cH}{{\mathcal{H}}}
\newcommand{\cI}{{\mathcal{I}}}
\newcommand{\cJ}{{\mathcal{J}}}

\newcommand{\cN}{{\mathcal{N}}}

\newcommand{\al}{\alpha}
\newcommand{\be}{\beta}
\newcommand{\ga}{\gamma}
\newcommand{\Ga}{\Gamma}
\newcommand{\de}{\delta}

\newcommand{\la}{\lambda}
\newcommand{\La}{\Lambda}

\newcommand{\si}{\sigma}

\newcommand{\om}{\omega}
\newcommand{\Om}{\Omega}

\newcommand{\ze}{\zeta}

\newcommand{\R}{\mathbb R}

\newcommand{\bH}{\mathbb H}

\newcommand{\sub}{\subset}

\newcommand{\nin}{\notin}

\newcommand{\ol}{\overline}

\newcommand{\dis}{\displaystyle}

\newcommand{\Rar}{\Rightarrow}
\newcommand{\lb}{\left(}
\newcommand{\rb}{\right)}
\newcommand{\ten}{\otimes}              % Tensor product %
\newcommand{\OL}{\Omega^1_{B|A}}
\newcommand{\oml}{\omega^1_{B|A}}
\newcommand{\omk} {\omega^1_{A}}
\newcommand{\OK} {\Omega^1_{A}}
\newcommand{\js} {\mathcal{J}_{\si}}
\newcommand{\is} {\mathcal{I}_{\si}}

\newcommand{\ns}{\mathcal{N}_{\si}}

\newcommand{\dn} {\overline{\Delta_N}}
\newcommand{\I}{\mathbb{I}}
\newcommand{\z}{\mathfrak{z}}
\newcommand{\m}{\mathfrak{m}}
\newcommand{\dlo}{\delta \log} 
\newcommand{\T}{\Theta}
\DeclareMathOperator{\tr}{Tr}

\DeclareMathOperator{\dl}{dlog} 
\DeclareMathOperator{\rsw}{rsw} 
\DeclareMathOperator{\Gal}{Gal}
\DeclareMathOperator{\Sw}{Sw}
\DeclareMathOperator{\Art}{Art}

\makeatletter
\let\c@equation\c@thm
\makeatother
\numberwithin{equation}{section}

\begin{document}

\title{RAMIFICATION THEORY FOR DEGREE $p$ EXTENSIONS OF ARBITRARY VALUATION RINGS IN MIXED CHARACTERISTIC $(0,p)$}

\author{Vaidehee Thatte}
\date{June 23, 2017}
\maketitle
\begin{abstract} 

We previously obtained a generalization and refinement of results about the ramification theory of Artin-Schreier extensions of discretely valued fields in characteristic $p$ with perfect residue fields to the case
of fields with more general valuations and residue fields. As seen in \cite{V}, the ``defect" case gives rise to many interesting complications. \\  In this paper, we present analogous  results  for degree $p$ extensions of arbitrary valuation rings in mixed characteristic $(0,p)$ in a more general setting.  More specifically, the only assumption here is that the base field $K$  is  henselian. In particular, these results are true for defect extensions even if the rank of the valuation is greater than $1$. A similar method also works in equal characteristic, generalizing the results of \cite{V}. 

\end{abstract}

\tableofcontents
\newpage
\section{Introduction}

Let $K$ be a henselian valued field of mixed characteristic $(0,p)$ with arbitrary valuation and $L|K$ a non-trivial Galois extension of degree $p$. We present a generalization and refinement of the classical ramification theory in this case. In \cite{V}, we considered Artin-Schreier extensions, when the defect is trivial or the valuation is of rank $1$. The results we present in this paper are true without such assumptions. We also remark that similar methods can be used to improve the results of \cite{V} and it is possible to remove the aforementioned assumptions.

First we consider Kummer extensions $L|K$, where $K$ contains a primitive $p^{th}$ root $\ze$ of unity.  The general case is then reduced to this case, by using tame extensions and Galois invariance.

\subsection{Invariants of Ramification Theory} 
Let $K$ be a valued field of characteristic $0$ with henselian
valuation ring $A$,  valuation $v$ and residue field
$k$ of characteristic $p>0$. We assume that $K$ contains a primitive $p^{th}$ root of $1$, let us denote it by $\ze_p=\ze$. Let $L=K(\alpha)$ be the (non-trivial) Kummer extension defined
by $\al^p =h$ for some $h \in K^{\times} $. For   any $a \in K^{\times}$, $h$ and $ha^p$ give rise to the same extension $L$.   Let $B$ be the
integral closure of $A$ in $L$. Since $A$ is henselian, it
follows that $B$ is a valuation ring. Let $w$ be the unique valuation
on $L$ that extends $v$ and let $l$ denote the residue field
of $L$. We denote  the value group of
$K$ by $\Ga_K :=v(K^{\times})$. The Galois group $\Gal(L|K)=G$ is cyclic of order $p$,
generated by $ \si : \al \mapsto \ze\al$. Let $\z:= \ze -1$

Let $\mathfrak{A} = \{h \in K \ \mid $ the solutions of
the equation $\al^p=h$ generate $L$ over $K \} $.
Consider the ideals $\js$ and $\cH$, of $B$ and $A$ respectively,
defined as below:
\begin{equation} 
\js = \ \lb \left\{ \frac{\si(b)}{b}-1 \mid b \in L^{\times}
\right\} \rb \sub B
\end{equation} 

\begin{equation}
\cH= \ \lb \left\{ \frac{\z^p}{h-1} \mid h \in \mathfrak{A} \right\}
\rb \sub A
\end{equation}

It is not apparent from the definition that $\cH$ is indeed a subset of $A$, we prove that in \Cref{H}. Our first result compares these two invariants via the norm map
$N_{L|K}=N$, by considering the ideal $\ns$ of $A$ generated by
the elements of $N(\js)$.
We also consider the ideal $\dis \is = \lb \{ \si(b) -b \mid b
\in B \} \rb$ of $B$. The ideals $\is$ and $\js$ play the roles
of $i(\si)$ and $j(\si)$ (the Lefschetz numbers in the classical
case, as explained in 2.2), respectively, in the generalization.
\subsection{Main Results} We will prove the following results in sections 4 and  6, respectively.  Then extend them to the non-Kummer case, in section 7. 

\begin{thm}\label{hn} If $L|K$ is as in 1.1, we have
the following equality of ideals of $A$:
\begin{equation}
  \cH = \ns
\end{equation}
\end{thm}
\begin{thm}\label{comm dia} For $L|K$  as in 1.1, we
consider the $A$-module $\omk$ of logarithmic differential
$1$-forms and the $B$-module $\oml$ of relative logarithmic differential
$1$-forms. Then
\begin{enumerate}[(i)]
\item There exists a unique homomorphism of $A$-modules 
$\dis \rsw : \cH/\cH^2 \to \omk/(\is \cap A)\omk$ such that for all $ h \in
\mathfrak{A}$, $\dis
\frac{\z^p}{h-1} \mapsto \frac{1}{h-1} \dl h$. 
\item There is a $B$- module isomorphism $\dis \varphi_\si:
\oml/\js\oml \overset{\cong}{\to }\js/\js^2$ such that for all $ x \in L^{\times},  \dis \dl x \mapsto \frac{\si(x)}{x}-1.$ 

\item Furthermore, these maps induce the following commutative
diagram:

\begin{center}$
\begin{tikzcd}[column sep=large]
\oml/\js \oml \arrow{r}{\varphi_\si}[swap]{\cong}
\arrow{d}[swap]{\dn}
&\js/\js^2  \arrow[hook]{d}{\overline{N}}\\
\omk/(\is \cap A )\omk   &\cH/\cH^2 \arrow{l}{\rsw}
\end{tikzcd}$
\end{center}

\noindent The maps $\dn, \overline{N}$ are induced by the norm
map $N$.
\end{enumerate}
\end{thm}

The map $\rsw$ in (i) is a refined generalization of the refined
Swan conductor of Kato for complete discrete valuation rings
\cite{K2}.
\begin{rem} If $L|K$ is unramified ($e_{L|K}=1, l|k$ separable of degree $p$), then we have $i(\si)=j(\si)=0$,  $\is=\js=B$ and $\cH=A$. Consequently, our main results are trivially true. From now on, we assume that $L|K$ is either wild ($e_{L|K}=p, l|k$ trivial ), ferocious ($l|k$ purely inseparable of degree $p$) or with defect.
\end{rem}

\subsection{Outline of the Contents} We begin, in section 2, with a preliminary discussion of K{\"a}hler differentials, defect and classical invariants of ramification theory. Section 3 contains the description of Swan conductor in the defectless case and some results that connect defect with the ideal $\js$.\\ We prove \Cref{hn} in section 4. In the next section, we use it to prove \Cref{fil}. This allows us, in the defect case, to express the ring $B$ as a filtered union of rings of the form $A[x]|A$, where the elements $x \in L^{\times}$ are chosen in a particular way.\\ The generalized and refined definition of the refined Swan conductor $\rsw$ is presented in section 6. First we define it in the defectless case and then extend the definition to  defect extensions. We also prove \Cref{comm dia}, first for  defectless extensions and then for defect extensions, using \Cref{fil}.
Results that can be proved in a manner similar to the Artin-Schreier case are presented without proofs.\\
In the seventh section, we extend the main results to the non-Kummer case.\\ The last section consists of some remarks about how the results of \cite{V} can be generalized to Artin-Schreier defect extensions of higher rank valuations, in a similar fashion.

\section{Preliminaries}
\subsection{Definitions}
\begin{defn} \textbf{Differential $1$-Forms}

\begin{enumerate}[(i)]
\item Let $R$ be a commutative ring. The $R$-module $\Om_R^1$ of
\tit{differential $1$-forms over $R$} is defined as follows:
$\Om_R^1$ is generated by
\begin{itemize}
\item The set $\{ db \mid b \in R \}$ of generators.
\item The relations are the usual rules of differentiation:
For all $b, c \in R$,
\begin{enumerate}
\item (Additivity) $d(b+c)=db+dc$
\item (Leibniz rule) $d(bc)=cdb+bdc$
\end{enumerate}
\end{itemize}
\item For a commutative ring $A$ and a commutative $A$-algebra
$B$, the $B$-module $\OL$ of \tit{relative differential
$1$-forms over $A$} is defined to be the cokernel of the map $B
\ten_A \Om^1_A \to \Om^1_B.$
\end{enumerate}
\end{defn}

\begin{defn} \textbf{Logarithmic Differential $1$-Forms}
\begin{enumerate}[(i)]
\item For a valuation ring $A$ with the field of fractions $K$,
we define the \textit{$A$-module $\omk$ of logarithmic
differential $1$-forms} as follows: $\omk$ is generated by
\begin{itemize}
\item The set $\{ db \mid b \in A \} \cup \{ \dl x \mid x
\in K^{\times}\}$ of generators.
\item The relations are the usual rules of differentiation and
an additional rule: For all $b,c \in A$ and for all $x,y \in
K^{\times},$
\begin{enumerate}
\item (Additivity) $d(b+c)=db+dc$
\item (Leibniz rule) $d(bc)=cdb+bdc$
\item (Log 1) $\dl(xy)=\dl x+\dl y$
\item (Log 2) $b \dl b = db$ for all $0 \neq b \in A$
\end{enumerate}
\end{itemize}
\item Let $L|K$ be an extension of henselian valued fields, $B$
the integral closure of $A$ in $L$ and hence, a valuation ring.
We define the \textit{$B$-module $\oml$ of logarithmic relative
differential $1$-forms over $A$} to be the cokernel of the map
$B \ten_A \om^1_A \to \om^1_B.$

\end{enumerate}
\end{defn}

\begin{defn}\textbf{Defect}

 Let $K$ be a henselian valued field of mixed characteristic $(0,p)$ and $L|K$ a non-trivial Galois extension of degree $p>0$. Let  $e_{L|K}:=(w(L^{\times}):w(K^{\times}))$ denote the ramification index and   $f_{L|K}:= [l:k]$  the inertia degree of $L|K$. Then $p=d_{L|K}e_{L|K}f_{L|K}$, where $d_{L|K}$ is a positive integer, called the \tit{defect} of the extension. Since $p$ is a prime, $d_{L|K}$ is either $1$ or $p$.
\end{defn}

For a more general discussion on defect, see \cite{F}.
\subsection{Classical Invariants}
Let $K$ be a complete discrete valued field of residue
characteristic $p>0$ with normalized  valuation $v$,
valuation ring $A$ and perfect residue field $k$. Consider
$L|K$, a finite Galois extension of $K$. Let $e_{L|K}$ be
the ramification index of $L|K$ and $G= \Gal (L|K)$. Let $w$ be
the valuation on $L$ that extends $v$, $B$ the integral
closure of $A$ in $L$ and $l$ the residue field of $L$. In this
case, we have the following invariants of ramification theory:
\begin{itemize}
\item The Lefschetz number $i(\si)$ and the logarithmic Lefschetz
number $j(\si)$ for $\si \in G\backslash \{1 \}$ are defined as
$$
i(\si)=\min\{v_L(\si(a)-a)\mid a \in B \}
$$$$
j(\si)=\min \left\{v_L\lb \frac{\si(a)}{a} -1 \rb\mid a \in L^{\times} \right\}
$$
Both the numbers are non-negative integers.
\item For a finite dimensional representation $\rho$ of $G$ over
a field of characteristic zero, the Artin conductor $\Art(\rho)$
and the Swan conductor $\Sw(\rho)$ are defined as
$$
\Art(\rho)= \frac{1}{e_{L|K}} \sum_{\si \in G\backslash \{1 \}}
i(\si)(\dim(\rho)-\tr(\rho(\si)))
$$$$
\Sw(\rho)= \frac{1}{e_{L|K}} \sum_{\si \in G\backslash \{1 \}}
j(\si)(\dim(\rho)-\tr(\rho(\si)))
$$
Both these conductors are non-negative integers. This is a consequence of the
Hasse-Arf Theorem (see \cite{S} chapters 4, 6).
\end{itemize}

\noindent The invariants $j(\si)$ and $\Sw(\rho)$ are the parts of
$i(\si)$ and $\Art(\rho)$, respectively, which handle the wild
ramification. We wish to generalize these concepts to arbitrary valuation rings. Let us begin with the case of discrete valuation rings,
possibly with imperfect residue fields.

\section{Swan Conductor, Best $h$ and Defect}
\subsection{Complete Discrete Valuation Case} The following lemma classifies Kummer extensions of complete discrete valued fields.
\begin{lm}\label{cdvr classifn}(See \cite{OH}, \cite{XZ}.)\\Let $L|K$ be an extension of complete discrete valued fields, $\pi$ a prime element of $K$. We use the notation of 1.1, $L=K(\al)$ is given by $\al^p=h$. We can choose $h$ with either $v(h)=1$ or $v(h)=0$ such that $t=v(h-1)$ is maximal. Then we have the following cases: $ u, c \in A^{\times}, \ol{u} \nin k^p, e'=\frac{v(p)}{p-1}=v(\z).$
\begin{enumerate}[(i)]
\item $h=1+c \z^p; \ol{c} \nin \{x^p-x \mid x \in k \}$.
\item $h=c\pi$.
\item $h=1+c\pi^t; 0 <t<e'p, (t,p)=1$.
\item $h=u$.
\item $h=1+u\pi^{t}; 0<t<e'p, p \mid t$. 
\end{enumerate}
\end{lm}
In the case (i), $L|K$ is unramified. In (ii) and (iii), it is wild  and in the last two cases, it is ferocious. We compute $j(\si)$ in each case. 
\begin{enumerate}[(i)]
\item $i (\si)=j(\si)= w\lb (\si-1)\lb\frac{\z}{\al-1}\rb\rb=0$.
\item $j(\si)=w\lb\frac{\si(\al)}{\al}-1\rb=w(\z)=e'p$.
\item $j(\si)=w\lb\frac{\si(\al-1)}{\al-1}-1\rb=w(\z)-w(\al-1)=e'p-t$.
\item $j(\si)=w\lb\frac{\si(\al)}{\al}-1\rb=w(\z)=e'$.
\item $j(\si)=w\lb\frac{\si(\al-1)}{\al-1}-1\rb=w(\z)-w(\al-1)=e'- t/p$.
\end{enumerate}

\subsection{Best $h$ and Swan Conductor: Classical Case and General Case}
\begin{defn} Let $L|K$ be as in \Cref{cdvr classifn}. We do not require $k$ to be perfect. We define the Swan conductor of this extension by 
\begin{equation} \dis \Sw(L|K):= \min_{h \in \mathfrak{A}} v\lb\frac{\z^p}{h-1}\rb
\end{equation}
This definition coincides with the classical definition of $\Sw(L|K)$ when $k$ is perfect.\\Any element $h$ of $\mathfrak{A}$ that achieves this minimum value is called \tit{best $h$}.\\It is well-defined upto multiplication by $a^p; a \in K^{\times}$.
\end{defn}
\begin{rem} \Cref{cdvr classifn} explicitly describes best $h$. $\Sw(L|K)$ is $0$ in (i), $e'p$ in (ii), (iv) and $e'p-t$ in (iii), (v).
\end{rem}
\noindent We generalize the definition of best $h$ to arbitrary extensions as in 1.1.
\begin{defn} Let $L|K$ be as in 1.1. An element $h$ of $\mathfrak{A}$ is called \tit{best} if \begin{equation}\dis  v\lb\frac{\z^p}{h-1}\rb=\inf_{g \in \mathfrak{A}} v\lb\frac{\z^p}{g-1}\rb\end{equation} If $h$ is best, $\cH$ is the principal ideal generated by $\lb\frac{\z^p}{h-1}\rb$ and plays the role of $\Sw(L|K)$ in the generalization. We cannot, however, guarantee the existence of best $h$ in general.
\end{defn} 

\subsection{Defect, $\js$ and Best $h$}

\begin{lm}\label{mu} Let $L|K$ be as in 1.1, except that we don't require $ \ze \in K$. Assume further that $L|K$ is either wild or ferocious. Then 
\begin{enumerate} 
\item There exists $\mu \in L^{\times}$ such that either $w(L^{\times})/w(K^{\times})$ is of order $p$ and generated by $w(\mu)$ or $l|k$ is purely inseparable of degree $p$ and generated by the residue class $\ol{\mu}$ of $\mu$.  
\item Let $\mu$ be as above, $x_i \in K$ for $0 \leq i \leq p-1$. Then $\dis \sum_{i=0}^{p-1}x_i\mu^i \in B$ if and only if $ x_i\mu^i \in B$ for all $i$.
\item $\dl \mu$ generates the $B$-module $\oml$.
\end{enumerate}
\end{lm} \begin{proof} See Lemma 1.11, Lemma 1.12, Lemma 1.13 of \cite{V}.  \end{proof}

\begin{pr}\label{js} Let $L|K$ be as in 1.1, except that we don't require $ \ze \in K$. Then $\js$ is a principal ideal of $B$ if and only if $L|K$ is defectless.\end{pr}
\begin{proof} See the proof of Proposition 3.10 of \cite{V}. \end{proof}

\begin{lm}\label{H} For  $L|K$  as in 1.1, 
$\cH$ is an integral ideal of $A$.
\end{lm}
\begin{proof} Let $h \in \mathfrak{A}, a \in \m_A$ such that  $h-1=a\z^p$ and set $\be=(\al-1) / \z$. Recall that $p$ divides $\z^p$ and $p\z^2$ divides $\z^p+p\z$. For some $x, x' \in B$, we have  $1+a\z^p=(1+\z \be)^p=1+p\z\be+\z^p\be^p+p\z^2\be^2 x=1+\z^p\be^p-\z^p\be+\z^p\be+ p\z\be+p\z^2\be^2x=1+\z^p(\be^p-\be)+p\z^2\be x'$.  Hence, $a=p\z^{(2-p)}\be x'+\be^p-\be$. Since $K$ is henselian, $\be \in K$ and this contradicts our assumption that $L|K$ is non-trivial.
\end{proof}
\begin{lm}\label{best} Let $L|K$ be as in 1.1. If $L|K$ is defectless, then we can find best $h$  satisfying exactly one of the following properties:
\begin{enumerate}[(i)]
\item $h=1+c \z^p; \ol{c} \nin \{x^p-x \mid x \in k \}$.
\item $h=ct; p \nmid v(t) $.
\item $h=1+ct; 0 <v(t)<e'p, p \nmid v(t) $.
\item $h=u$.
\item $h=1+ut; 0<v(t)<e'p, p\mid v(t)$. 
\end{enumerate}
where $t \in \m_{A}, u, c \in A^{\times}, \ol{u} \nin k^p$.
\end{lm}
\begin{proof} Follows from \Cref{js}.
\end{proof}

\section{Proof of \Cref{hn}}\noindent Let $L|K$ be as in 1.1. 
First we prove $\cH \sub \ns$.\\
Let $ h \in \mathfrak{A}$, we want to show that $\lb \frac{\z^p}{h-1} \rb A \sub \ns$. We observe that  $N(\al)=(-1)^{p-1}h$ and  $ N(\al-1)=  \prod_{i=0}^{p-1} (\ze^i\al-1) = (-1)^{p-1}(\al^p-1)=(-1)^{p-1}(h-1)$.\\If $v(h)>0, \lb \frac{\z^p}{h-1} \rb A= (\z^p)A$. Note that  $N\lb \frac{\si (\al)}{\al}-1 \rb =N(\z)=\z^p$.\\If $v(h)=0$, consider
$
N\lb \frac{\si(\al-1)}{\al-1}-1\rb=N\lb \frac{\z \al}{\al-1}\rb=\frac{\z^p h}{h-1}$. Since $h$ is a unit, $\frac{\z^p h}{h-1}$ and $\frac{\z^p}{h-1}$ generate the same ideal of $A$. Thus, it follows that $\cH$ is a subset of $\ns$.\\\\
Next, we prove the reverse inclusion $\ns \sub \cH$. If $L|K$
is defectless, this follows directly from \Cref{js} and \Cref{best}.
Proof in the defect case requires some work.\\\\
Let $L|K$ be a defect extension as in 1.1. The value group
$\Ga=\Ga_K$ need not be an ordered subgroup of
$\R$. Let $v$ denote the valuation on $L$ and also on $K$. Given any $b \in L \backslash K$, we want to show that $N\lb \frac{\si(b)- b}{b} \rb \in \cH$. It is enough to consider the case when $b$ is a unit. For any $\al$ such that $\al^p=h$ generates $K(\al)=L$, 
\begin{equation} N\lb \frac{\si(\al)}{\al}-1 \rb = N(\ze -1 )=\z^p \in \cH \end{equation}
If $v(\si(b)-b)= v\lb \frac{\si(b)- b}{b} \rb \geq v(\z)$, then $N\lb \frac{\si(b)- b}{b} \rb \in (\z^p) \sub \cH.$ Thus, we may assume
\begin{equation} v(\si(b)-b) < v(\z) \end{equation}
We divide the proof into two cases: $p=2$ and $p>2$.
\subsection{Case $p=2$} \begin{proof}In this case, $\si^2=id, \ze=-1$ and  $\z=-2$.
Let $b \in B^{\times} \backslash A$.\\ Since $v(\si(b)-b) < v(2)=v(2b), v(\tr(b))=v(\si(b)-b+2b)=v(\si(b)-b)$. Define $x_b=x:= \frac{\si(b)-b}{\tr(b)} \in B^{\times}$. Clearly, $\si(x)=-x$ and hence, $x^2=-N(x)=h \in A^{\times}$. As $\si$ does not fix $x, L=K(x).$\\ We have $x-1=\frac{-2b}{\tr(b)}\Rar \frac{\si(x-1)}{x-1}-1= \frac{\si(b)}{b}-1$.\\ Therefore, $ N\lb \frac{\si(b)}{b}-1\rb=N\lb \frac{\si(x-1)}{x-1}-1 \rb= N\lb \frac{-2x}{x-1}\rb=\frac{\z^2h}{h-1}$ is an element of $\cH$.\end{proof}

\subsection{Case $p>2$} \begin{proof} Consider the formal expression $\dis \tr_{L|K}=\tr= \frac{\si^p-1}{\si-1}=\prod_{i=1}^{p-1} (\si- \ze^i)$. Given $b  \in B^{\times} \backslash A$, define \\$\dis \ga_b=\ga:=\lb \frac{b^{p-1}}{g'(b)} \rb$ and $\dis y_b=y:= \lb \prod_{i=2}^{p-1} (\si- \ze^i)\rb(\ga) $; where $g(T)=\min_K(b)$. \\Then $(\si-\ze)(y)=\tr(\ga)=1$, i.e., $\si (y)=\ze y+1$. \\ Next define $x_b=x:=1+\z y$.  Then $\si (x)= 1+\z (\ze y+1)=1+\z + \ze (\z y)= \ze (1+\z y)=\ze x$.\\ Consequently, $x^p=N(x)=: h \in K$ and $L=K(x)$. We compare $v\lb\frac{\si(b)}{b}-1\rb=v(\si(b)-b)=:s$  and $v\lb\frac{\si(x-1)}{x-1}-1\rb= v\lb\frac{\si(\z y)}{\z y}-1\rb=v\lb\frac{\si(y)}{ y}-1\rb =v\lb\frac{1+\z y}{ y}\rb =:s'$.\\

\noindent For $1 \leq i \leq p-1$, let
$\dis v\lb \lb \prod_{j=p-i}^{p-1} (\si- \ze^{j})\rb(\ga) \rb=v\lb \lb \prod_{j=p-i+1}^{p-1} (\si- \ze^{j})\rb(\ga) \rb+c_i$; $c_i \geq 0$. Then
\begin{enumerate}[(i)]
\item $v(\ga)=-(p-1)s\Rar \sum_{i=1}^{p-1}c_i=(p-1)s$.
\item $\dis\sum_{i=1}^{p-2}c_i=v(y)-v(\ga)=v(y)+(p-1)s$ 
\item  $c_{p-1}=v((\si-\ze)(y))-v(y)=-v(y)$. 
\end{enumerate}

\noindent Since $v(\z)>s$, we have $v(\z y) >0 \Rar v(x)=v(1+ \z y) =0, -v(y)=s'=c_{p-1}$.\\
Consider $B_b:= A[ \{ \si^i(b) \mid 1 \leq i \leq p-1 \}] \subset B$.\\ It is invariant under the action of $\si^i - \ze^j$ for all $1 \leq i \leq p-1, 0 \leq j \leq p-1$.\\
For the ideal ${\js}_{,b}
:= \lb \{\si(t)-t \mid t \in B_b \} \rb$ of $B_b,$ the ideal $ {\js}_{,b}B$ of $B$ is finitely generated and therefore, principal.\\
Observe that
$$\ga = \frac{b^{p-1} N(g'(b))/g'(b)}{N(g'(b))} \in \lb \frac{1}{N(g'(b))} \rb B_b $$

\noindent Therefore, 
 $c_i \geq s= v(\si(b)-b)$ for all $1 \leq i \leq p-2$.  \\Consequently, we have $s+v(y)=s-s'=s-c_{p-1} \geq 0$\\
$ \Rar \lb\frac{\si(b)}{b}-1\rb \sub  \lb\frac{\si(x-1)}{x-1}-1\rb \\ \Rar \lb N\lb\frac{\si(b)}{b}-1\rb \rb \sub  \lb N \lb\frac{\si(x-1)}{x-1}-1\rb \rb= \lb \frac{\z^ph}{h-1}\rb= \lb \frac{\z^p}{h-1}\rb \sub \cH$.\\ This concludes the proof.

\end{proof}

\begin{rem} In \cite{V}, we used an  argument that required the rank of the valuation to be $1$. The above argument, however, works for valuations of arbitrary rank.

\end{rem}
\begin{cor} For $L|K$ as in 1.1, the following statements are equivalent:
\begin{enumerate}
\item Best $h$ exists.
\item $\cH$ is a principal ideal of $A$.
\item $\js$ is a principal ideal of $B$.
\item $L|K$ is defectless.
\end{enumerate}
\end{cor}

\section{Filtered Union in the Defect Case}
Let $L|K$ be a defect extension  as in 1.1. We will write the ring $B$ as a filtered union of rings $A[x]$ and study the extensions $K(x)|K$ for a better understanding of $L|K$.

Let $\mfrak{A'}:=\{ h \in \mfrak{A} \mid h \in A^{\times}, h-1 \in \m_A\}$. We note that in the defect case, $\dis \cH= \lb \left\{ \frac{\z^p}{h-1} \mid h \in \mathfrak{A'} \right\}
\rb$.
\begin{thm}\label{fil} Consider $\dis \mathscr{S}= \{ \al \in L
\mid \al^p=h \in \mfrak{A'} \}$. For
each $\al \in \mathscr{S}$, we can find $\dis \al' \in
B^{\times} \cap \al K^{\times} $ such that $B= \cup_{\al \in
\mathscr{S}} A[\al']$ is a filtered union, that is, the
following are true:
\begin{enumerate}[(i)]
\item For any $\al_1, \al_2 \in \mathscr{S}$, either $A[\al_1']
\sub A[\al_2']$ or $A[\al_2'] \sub A[\al_1']$.
\item Given any $\be \in B$, there exists $\al \in \mathscr{S}$
such that $\be \in A[\al']$.
\end{enumerate}
\end{thm}

\begin{defn} For $\al \in \mathscr{S}$, define $\al':= \ga \frac{(\al-1)}{\z}$; where $\ga_{\al}=\ga \in A$ such that $\al' \in B^{\times}$. We will show that these $\al$'s satisfy the conditions of \Cref{fil}. Note that the ring $A[\al']$ does not depend on the choice of $\ga$.
\end{defn}
\subsection{Preparation for the Proof} 
\begin{lm} If  $\al_1, \al_2 \in \mathscr{S}$ such that $v(\al_1-1) \leq v(\al_2-1), $ then $A[\al_1'] \sub A[\al_2']$.
\end{lm}
\begin{proof}Since $\al_1, \al_2 \in B^{\times}$ generate the same extension, (by \Cref{conj}) we have $\si \lb \frac{\al_1}{\al_2}\rb=\frac{\al_1}{\al_2}=:u \in K \cap B^{\times}=A^{\times}$.
\begin{align*}\dis
\al_1'&= \ga_1 \frac{(\al_1-1)}{\z} \\
&=\lb\frac{\ga_1}{\ga_2}\rb\ga_2 \lb\frac{\al_2 u -1}{\z}\rb \\
&= \lb\frac{\ga_1}{\ga_2}\rb\ga_2 \lb\frac{\al_2 u -u + u-1}{\z}\rb \\
&= \lb\frac{\ga_1}{\ga_2}\rb u \ \ga_2 \frac{(\al_2-1)}{\z}+ \lb\frac{\ga_1}{\ga_2}\rb\ga_2\lb\frac{u-1}{\z}\rb\\
&=\lb\frac{\ga_1}{\ga_2}\rb u \al_2'+ \ga_1\lb\frac{u-1}{\z}\rb
\end{align*}
\noindent  $v(\al_1-1) \leq v(\al_2-1) \Rar v(\ga_1)\geq v(\ga_2)$ and $u \in A^{\times}$. Hence, $\lb\frac{\ga_1}{\ga_2}\rb u \al_2' \in A[\al_2']$.\\ Furthermore, $v(u-1)=v(\al_1-\al_2)+v(\al_2)=v((\al_1-1)-(\al_2-1)) \geq v(\al_1-1)$.\\
 Hence, $v\lb \ga_1\lb\frac{u-1}{\z}\rb\rb \geq v(\al_1')=0$. Since $\ga_1\lb\frac{u-1}{\z}\rb \in K$, we see that $\al_1' \in A[\al_2']$.

\end{proof}

\begin{rem}\label{conj} Let $\si(\al_1)=\ze \al_1$ and $\si(\al_2)=\ze^i \al_2$ for some $1 \leq i \leq p-1$. Consider the unique number $j$ satisfying $1 \leq j \leq p-1$ and $ij \equiv 1 \mod p$. Clearly, $\al_2$ and $\al_2^j$ give the same extension $L|K$ and $\si(\al_2^j)=\ze^{ij}\al_2^j=\ze \al_2^j$. Thus, if $i \neq 1$, we can replace $\al_2$ by $\al_2^j$.

\end{rem}
\begin{lm}\label{bis} Given any $b \in B$, there exists $\al \in
\mathscr{S}$ such that $\lb \si(b)-b \rb \subset \lb
\si(\al')-\al'\rb $.
\end{lm}
\begin{proof} It is enough to consider the case $b \in B^{\times}$. Let  $\dis v_0= v(\si(b)-b) $. Since this is the defect case, $\is=\js$ and hence, by \Cref{hn},  $\cH= \ns = \lb \{N(\si(x)-x) \mid x \in B^{\times}\}\rb$. In particular, all the elements of valuation greater than or equal to $pv_0$ are in $\cH$. Pick some $\al \in \mathscr{S}, \al^p=h$ and let $c$ be as follows.

 $c:=v(\si(\al')-\al')$.  $\dis c=v(\si(\al')-\al')=v(\al \ga_\al)=v\lb\frac{\z}{\al-1}\rb=\frac{1}{p}v\lb\frac{\z^p}{h-1}\rb$. 

By definition of $\cH$, it is possible to choose $h$ such that $pc=v\lb\frac{\z^p}{h-1}\rb \leq pv_0$. Hence, there exists $\al \in
\mathscr{S}$ such that $c \leq v_0$.
\end{proof}
\begin{lm}\label{bin} For $x,y \in L$, we have $(\si-1)^n(xy)=
\sum_{k=0}^n {n \choose k}(\si-1)^{n-k}(x)(\si-1)^k(\si^{n-k}(y))
$\\
In particular, for $n=1$,
$(\si-1)(xy)=(\si-1)(x)\si(y)+x(\si-1)(y)$.
\end{lm}
\subsection{Proof of \Cref{fil}} As in the case of Artin-Schreier extensions (see section 5.4 of \cite{V}), it is enough to prove the following result:
\begin{pr}\label{btap} Given any $\be \in B^{\times}$, there
exists $\al \in \mathscr{S}$ such that\\
$\lb \prod_{i=1}^{p-1}(\si-\ze^i)\rb\lb\frac{1}{F'(\al')} A[\al', \be]\rb=\tr\lb\frac{1}{F'(\al')} A[\al', \be]\rb \sub B$. Here,
$F$ denotes the minimal polynomial of $\al'$ over $K$.
\end{pr}
\begin{proof} For any $\al \in \mathscr{S}$ with $\al^p=h$ , we have 
\begin{enumerate}[(i)]
\item $ \si (\al')=\ga \frac{(\al \ze -1)}{\z}=\ga \frac{(\al \ze - \ze + \ze- 1)}{\z}=\ze \al'+\ga$
\item $(\si-1)
(\al')= \z \al' +\ga=\ga \al$. Note that $v(\ga) < v(\z)=v(\ga)+v(\al-1)$.
\item $F' (\al')=\prod_{i=1}^{p-1}(\al'-\si^i(\al'))=\prod_{i=1}^{p-1}\frac{\ga }{\z}(\al-\si^i(\al))=p\lb \frac{\ga \al}{\z}\rb^{p-1}=\lb \frac{hp}{\al}\rb\lb \frac{\ga}{\z}\rb^{p-1}$. 
\item Since $v(p)=v(\z^{p-1}), v(F'(\al'))=v(\ga^{p-1})$.
\item For any $1 \leq i \leq p-1, \si^i\lb \frac{1}{F'(\al')}\rb= \ze^i\lb \frac{1}{F'(\al')}\rb$
\end{enumerate}
We want to show that for a ``special" $\al$, for all $0 \leq m, j \leq p-1$,
\begin{equation} v \lb  \lb \prod_{i=1}^{p-1}(\si-\ze^i)\rb\lb\frac{1}{F'(\al')} \al'^i \be^j\rb\rb \geq 0
\end{equation}
For any $\dis x \in L, 1 \leq i \leq p-1,\\ (\si - \ze^i) \lb \frac{x}{F'(\al')}\rb=\frac{\si (x) \ze}{F'(\al')}-\frac{x\ze^i }{F'(\al')}=\frac{\ze}{F'(\al')} (\si - \ze^{i-1})(x)$ by (v) above.\\ Thus, (5.8) is equivalent to 
\begin{equation} v \lb\ze^{p-1} \lb \prod_{i=1}^{p-1}(\si-\ze^{i-1})\rb \lb \al'^m \be^j\rb\rb= v \lb \lb \prod_{i=1}^{p-1}(\si-\ze^{i-1})\rb \lb \al'^m \be^j\rb\rb \geq v(F'(\al'))=(p-1)v(\ga)
\end{equation}
Since $\prod_{i=1}^{p-1}(\si-\ze^{i-1})=\prod_{i=1}^{p-1}(\si-1-(\ze^{i-1}-1))$ and $v(\ga)< v(\z) \leq v(\ze^{i-1}-1)$, it is enough to show 
\begin{equation} v((\si-1)^{p-1}(\al'^m \be^j)) \geq
v(F'(\al'))=(p-1)v(\ga)
\end{equation}
The rest follows from \Cref{bis}, \Cref{bin} and the following argument. This is taken directly from  \cite{V}, it is worth noting that we did not use the rank $1$ assumption in these steps and therefore, the argument is valid for higher rank valuations.

\color{blue}

% Can be deleted later if found unnecessary%

\begin{enumerate}[(Step 1)]
\item \textbf{Construction of the special $\al'$}\\We begin with
an $\al_0$ satisfying $\lb \si(\be)-\be \rb \sub \lb
\si(\al_0')-\al_0' \rb $. Let $(\si-1)(\be)=b_1\ga_0; b_1 \in B$.
Therefore, $(\si-1)^2(\be)=(\si-1)(b_1)\ga_0$. We don't know much
about the valuation of $(\si-1)(b_1)$, however. Let $\al_1$ be
such that $\lb (\si-1)(b_1) \rb \subset \lb (\si-1)(\al_1') \rb$.
Write $(\si-1)(b_1)=b_2 \ga_1; b_2 \in B$. Now we can write
$(\si-1)^2(\be)= b_2 \ga_1\ga_0$. Using this process, we can find
$b_i$'s and $\al_i$'s such that $(\si-1)^i(\be)= b_i
\ga_{i-1}...\ga_1\ga_0;$ where $b_i \in B$.
\\Let $\ga$ be the $\ga_j$ with smallest valuation involved in the
expression for $i=p-1$. Let $\al$ denote the corresponding
$\al_j$. We will show that this $\al$ satisfies the required
property.
\item \textbf{Proof for $\be$}\\$\lb \si(\be)-\be \rb \sub \lb
\si(\al_0')-\al_0' \rb \subset \lb \si(\al')-\al' \rb=\lb \ga \rb$,
since $v(\ga) \leq v(\ga_0)$.
Due to the choice of $\ga$, we also have $v((\si-1)^t(\be)) \geq t
v(\ga)$ for all $1 \leq t \leq p-1$. In particular, this is true
for $t=p-1$, proving the statement (5.10) for the case $i=0,
j=1$.
\item \textbf {Terms $\al'^m \be^j$}\\ For the terms of the form
$\be^j$, we use induction on $j$ and \Cref{bin}. Valuation of
each term in the expansion is at least $(p-1)v(\ga)$. In fact, by
a similar argument, $v((\si-1)^k(\be^j)) \geq k v(\ga)$ for all $1
\leq k \leq p-1$. \\ For the general terms $\al'^m \be^j$, first
note that $(\si-1)^k(\al')=(\si-1)^{k-1}(\ga)=0$ for all $k >1$.
Therefore, (again using the identity), we have\\
$(\si-1)^{p-1}(\al'^m \be^j)=\al'^m(\si-1)^{p-1}(\be^j)+
(p-1)(\si-1)(\al'^m)(\si-1)^{p-2}(\si(\be^j))$. Once again, both
these terms have valuation $\geq(p-1)v(\ga)$.
\end{enumerate}
\end{proof}

\section{Refined Swan Conductor and Proof of \Cref{comm dia} }

\begin{defn} Let $K$ be as in 1.1. For any $x \in K^{\times}$, we define elements $\dlo x \in \bH_x \ten \omk$ as described below. $\bH_x$ is the ideal of $A$ given by $\bH_x:= (x-1)A \cap A \cap \lb \frac{1}{x-1}\rb A$ if $x \neq 1$ and $\bH_1:=(0)$.
\begin{itemize}
\item $\dlo 1 := 0$.
\item If $0 \neq x \in \m_A, \dlo x:= 1 \ten \dl x$.
\item If $1 \neq x \in A^{\times}, \dlo x:= (x-1) \ten \frac{\dl (x-1)}{x}$.
\item If $ x \in K \backslash A, \dlo x:= -\dlo (1/x)$.
\end{itemize}
Furthermore, for any $0 \neq b \in A$, we define elements $\de b \in b\bH_b \ten \omk$ by $\de b:= b \dlo b$ and $\de 0:= 0$. 
\end{defn}
\begin{lm} For all $b,c \in A$ and for all $x,y \in K^{\times}$, we have 
\begin{enumerate}
\item (Log 1) $\dlo(xy)=\dlo x+\dlo y$ in $(\bH_x \cup \bH_y \cup \bH_{xy}) \ten \omk.$
\item (Additivity) $\de(b+c)=\de b+\de c$ in $(b\bH_b \cup c\bH_c \cup (b+c)\bH_{b+c}) \ten \omk.$
\item (Leibniz rule) $\de(bc)=c\de b+b\de c$ in $(b\bH_b \cup c\bH_c \cup (bc)\bH_{bc}) \ten \omk.$
\end{enumerate}
\end{lm}

\subsection{Refined Swan Conductor $\rsw$ in the Defectless Case} 
We first define the refined Swan conductor for the defectless case (below) and then extend the definition  to  the defect case.

\begin{defn}Let $L|K$ be as in 1.1 and  defectless, given by best $h$. Consider the ideal $\I$ of $A$ defined by $$\I:= \left\{ x \in K \mid v(x) \geq \lb \frac{p-1}{p}\rb  v \lb \frac{\z^p}{h-1} \rb \right\}$$
We note that this definition only depends on the valuation of $h-1$ and hence, is independent of the choice of best $h$. \\The \tit{refined Swan conductor} $\rsw$ of this extension is defined to be the $A$-homomorphism  $$\dlo h: \cH \to  \omk/\I \omk$$ given by $$\dis r \mapsto \frac{r}{\z^p}\dlo h.$$  We will show in  \Cref{rsw well-def} that this definition is independent of the choice of best $h$.
 \end{defn}
\begin{rem}We can also view $\rsw$ as an element of $\lb\frac{1}{\z^p}\rb\bH_h \ten \omk$. This definition is consistent with Kato's definition in \cite{K2}. We note that $\bH_h$ is independent of choice of best $h$.\end{rem}

\begin{lm}\label{rsw well-def} Let $L|K$ be defectless, given by best $h$.  
Then the refined Swan conductor of this extension, i.e., the $A$-homomorphism  $\dlo h$, is independent of the choice of $h$. 
\end{lm}
\begin{proof} Let $h$ and $a^ph$ be best; $0,1 \neq a \in A$. Then the difference between the two $A$-homomorphisms is given by $\dis \dlo a^ph- \dlo h=\dlo a^p=p \dlo a$. 
For an element $r$ of $\cH= \lb\frac{\z^p}{h-1}\rb=\lb\frac{\z^p}{a^ph-1}\rb$, we have

$$
 (p \dlo a) (r) =\begin{cases}
               \frac{r  }{\z^p}p \ten \dl a; \hfill \dis 0 \neq a \in \m_A\\
               \frac{r }{\z^p}p(a-1)  \ten \frac{\dl (a-1)}{a}; \hfill \dis 1 \neq a \in A^{\times} \\
               \end{cases}
$$
We wish to show that $ (p \dlo a) (r)$ belongs to $\I \omk$. Considering the formulas above and observing that $\frac{r  }{\z^p}p$ has the same valuation as $\frac{r  }{\z^p}p(a-1)$ when $0 \neq a \in \m_A$, it is enough to show that $\frac{r  }{\z^p}p(a-1) \in \I$ for $0,1 \neq a \in A$. 
$$ v\lb \frac{r  }{\z^p}p(a-1) \rb \geq v\lb \frac{1}{h-1}p(a-1)\rb = v(a-1)+v(p) - v(h-1)=v(a-1)+ (p-1)v(\z)-v(h-1)$$
Thus, it is enough to prove the inequality $$ v(a-1)+ (p-1)v(\z)-v(h-1)\geq \lb \frac{p-1}{p}\rb v\lb \frac{\z^p}{h-1}\rb=\frac{1}{p}v(h-1)+(p-1)v(\z) - v(h-1)$$
That is,  \begin{equation}\label{ah} v(a-1) \geq \frac{1}{p}v(h-1) \end{equation}

\noindent When $h-1 \in A^{\times},$ this follows simply from $v(a^p-1)=pv(a-1) \geq 0=v(h-1)$, since $a \in A$. 

\noindent When $h-1 \in \m_A, h \in A^{\times}$ and hence,  \Cref{ah} follows from
$$v(h-1)=v(a^ph-1)=v((a^p-1)h+h-1) \geq \min(v(a^p-1), v(h-1))$$
where the first equality is a consequence of $h$ and $a^ph$ both being best.

\end{proof}

\begin{cor} Let $L|K$ and $h$ be as above. Then the ideal $\I'$ of $A$ generated  by $$\left\{ p \lb \frac{a-1}{h-1}\rb  \mid a \in A, a^ph~ is~ best \right\}$$ is contained in the ideal $\I$. \end{cor}

\subsection{Proof of \Cref{comm dia}}
\begin{lm}  Let $L|K$ be as in 1.1. Then 
\begin{enumerate}[(a)]
\item The norm map $N_{L|K}=N$ induces the surjective ring homomorphism $N: B \to A/(\is \cap A )$.
\item The map $\varphi_{\si}: \oml/\js \oml \to \js/\js^2$ given by $\dl x \mapsto \frac{\si(x)}{x}-1$ is a surjective $B$-module homomorphism.
\end{enumerate}\end{lm}\begin{proof} See Lemmas 1.6 and 6.1 of \cite{V}.\end{proof}

\subsubsection{Defectess Case} Let $L|K$ be as in 1.1 and defectless, given by best $h$. 
\begin{lm}\label{hii}When $L|K$ is defectless, $\cH \subset \I \subset \is \cap A$    
\end{lm}
\begin{proof}  Let $ p w\lb \frac{\z}{\al -1}\rb =p\nu=w\lb \frac{\z^p}{h-1}\rb $. By definition, $\I =\left\{ x \in K \mid w(x) \geq (p-1) \nu \right\}$ and hence, contains $\cH$. Now we need to show that $\is$ contains all the elements $x$ of $L$ satisfying $w(x) \geq (p-1) \nu$. 
Using  the characterization in \Cref{best}, we see that in Case (i), $\nu=0$ and the result follows trivially. Thus, we may assume that $L|K$ is either wild or ferocious.
\\\\Without loss of generality, we may further assume $0 \leq v(h) < p v(\z)$.
We will divide the proof  in two cases:
\begin{itemize}
\item Case $p>2$: If $h \in \m_A, \nu=w(\z)$. Since $\al \in B, \z \al = (\si - 1) (\al) \in \is$. By our assumption on $h$, $w(\z \al) < 2w(\z) \leq (p-1)w(\z)$ and hence, $\I \subset \is \cap A.$\\ When $h \in A^{\times}, $ we consider the element $b$ of $\is$ given by $$b=(\si - 1)\lb \frac{\z}{\al -1} \rb = \lb \frac{-\z^2}{(\al -1)(\ze \al -1)} \rb.$$
Since $w(b)=2\nu \leq (p-1)\nu, \I \subset \is \cap A.$

\item Case $p=2$: In this case, $(p-1)=1$ and $\z=-2$. 
By \Cref{mu}, the ideal $\is$ of $B$ is generated by the elements $(\si-1)(x\mu)=x (\si-1)(\mu); x \in K, x\mu \in B$ where $\mu$ is either $\al$ or $\al-1$.
\\ Since $(\si-1)(\al-1)=(\si-1)(\al)=-2\al,$ $\is$ is generated by $(2\al )K \cap B.$\\
When $\al-1 \in B^{\times}, w\lb \frac{2}{\al-1} \rb =w(2)>w\lb \frac{2}{\al} \rb= w\lb \frac{2 \al}{h} \rb \geq 0.$ Since $\frac{2 \al}{h} \in \is, \I \subset \is \cap A$\\
If $\al - 1 \in \m_B$ and $e_{L|K}=1, \is = \js  = \lb \frac{2}{\al-1} \rb B$ and therefore, $\I \subset \is \cap A$.\\ 
The last remaining case is when $\al - 1 \in \m_B$ and $e_{L|K}=2$. As in the preceding case, $\I = \js \cap A$. Since $\is \neq \js$ in general, we use a different strategy.

Let $w(2)=2e', 0<w(h-1)=2s<4e',  w(\al-1)=s, \nu=2e'-s$ \\ 
If there exists an element $ b \in A$ such that $s < w(b) <2e',$ then $\frac{2 \al}{b} \in \is$ and $w\lb \frac{2\al}{b}\rb < \nu$. Now suppose that there is no such element. In particular, $2s \geq 2e'$.
Any element $x$ of $\I$ must satisfy $w(x)>2e'-s$. If there is an $x \in \I$ such that $2e'> w(x)>2e'-s$, then $2s> w\lb\frac{x (h-1)}{2} \rb >s$. By the assumption above, we must have $w\lb\frac{x (h-1)}{2} \rb \geq 2e'$ and hence, $w(x) \geq 2e'+2e'-2s=2 \nu$.\\

Since $$w\lb(\si -1)\lb \frac{2)}{\al-1} \rb \rb=w \lb \frac{-4\al}{h-1} \rb =4e'-2s=2 \nu,$$
we have $\I \subset \is$.

\end{itemize}
\end{proof}

\begin{proof}[Proof of \Cref{comm dia} in the defectless case]

We will use the characterization in \Cref{best}.
\begin{itemize}
\item  Case (i): $h=1+c\z^p; \ol{c} \nin \{ x^p-x \mid x \in K \}$. As mentioned earlier, $L|K$ is unramified in this case and the result is trivially true.
\item  Case (ii): In this case, $h \in \m_A$ and $p \nmid v(h)$. Consequently, the $B$-module $\oml$
 is generated by $\dl \al$ and the diagram is given by the following maps:
\begin{center}$
\begin{tikzcd}[column sep=large]
b\dl \al \arrow{r}{\varphi_\si}
\arrow{d}[swap]{\dn}
& b\z  \arrow{d}{\overline{N}}\\
N(b)\dl N(\al)= N(b)\dl h   & N(b)\z^p \arrow{l}{\rsw}
\end{tikzcd}$
\end{center}
For $b \in B$, $b \dl (\al) \in Ker(\varphi_\si) \iff b\z \in (\z^2)$. Hence, $\varphi_\si$ is an isomorphism. Since $\cH \subset \I \subset (\is \cap A)$,  the map $\rsw$ is well-defined, independent of the choice of $h$. 
\item Cases (iii)-(v): In these cases, $1 \neq h \in A^{\times}, \al \in B^{\times}$ and  the $B$-module $\oml$
 is generated by $\dl (\al-1)$. The diagram is given by the following maps:
\begin{center}$
\begin{tikzcd}[column sep=large]
b\dl (\al-1) \arrow{r}{\varphi_\si}
\arrow{d}[swap]{\dn}
& b \lb \frac{\z \al}{\al-1}\rb \arrow{d}{\overline{N}}\\
N(b)\dl N(\al-1)  & N(b)\lb \frac{\z^p h}{h-1}\rb \arrow{l}{\rsw}
\end{tikzcd}$
\end{center}
It is easy to verify that $\varphi_\si$ is an isomorphism.\\ Since $\cH \subset \I \subset (\is \cap A)$,  the map $\rsw$ is well-defined, independent of the choice of $h$. By definition, $$\rsw\lb N(b) \lb \frac{\z^p h}{h-1}\rb \rb:=N(b)\lb \frac{h}{h-1}\rb (h-1) \lb \frac{\dl (h-1)}{h}\rb=N(b)\dl (h-1)=N(b)\dl N(\al-1).$$ The rest follows.
\end{itemize}\end{proof}

\subsubsection{Preparation for the defect case} Let $L|K$ be a defect extension as in 1.1. Recall that $B= \cup_{\al \in
\mathscr{S}} A[\al']$ is a filtered union, where $\dis \mathscr{S}= \{ \al \in L
\mid \al^p=h \in A^{\times}, h-1 
\in \m_A \}$ and 
for each $\al \in \mathscr{S}$, we have $\ga_{\al}=\ga \in A$ such that $\al':= \ga \frac{(\al-1)}{\z} \in B^{\times}$. Since there is defect, we consider $\OL$ and $\OK$ instead of $\oml$ and $\omk$, respectively. Fix some $\al_0 \in \mathscr{S}$ as the starting point. Let $v(\al_0-1)-v(\z)=-v(\ga_0)=-\mu <0$. We may only consider the subset $\mathscr{S}_0:= \{ \al \in \mathscr{S} \mid v(\al -1) >v(\al_0 -1)\}$ of $\mathscr{S}$.
\begin{lm}\label{alpha level} Let $\al \in \mathscr{S}_0, \al^p=h_\al, \al_0^p=h_0$. Let $F_{\al}(X)$ and $F_0(X)$ denote the minimal polynomials over $K$ of $\al'$ and $\al'_0$, respectively. Consider 
$c_\al:=F'_{\al'}(\al')$, $c_0:=F'_{0}(\al'_0)$ and the ratio
$\al_0\ga_0/\al\ga_\al =: a_\al \in A$. Then we have the following
commutative diagram:

\begin{center} $\dis
\begin{tikzcd}[column sep=large]
\Om^1_{A[\al_0']|A} \arrow{r}{\cong} \arrow[hook]{d}{\rho_\al}
&A[\al_0']/(c_0) \arrow{r}{\cong} \arrow[hook]{d}{\iota_\al}
&(\frac{1}{a_0})A[\al_0']/(\frac{c_0}{a_0})A[\al_0']
\arrow[hook]{d}{j_\al}\\
\Om^1_{A[\al']|A} \arrow{r}{\cong} &A[\al']/(c_\al)
\arrow{r}{\cong}
&(\frac{1}{a_\al})A[\al']/(\frac{c_\al}{a_\al})A[\al']
\end{tikzcd}$ 
\end{center}

Here, the isomorphisms are given by $b_0 d\al'_0 \mapsto b_0 \mapsto b_0/a_0$ and  $b d\al' \mapsto b \mapsto b/a_\al$ for all $b_0 \in A[\al_0']$ and $b \in A[\al']$. The vertical maps are given by multiplication by $a_\al$.
\end{lm}
\begin{proof} Let us omit the subscript $\al$ for convenience. Since $\al$ and $\al_0$ give rise to the same extension, $\al_0/\al=:u$ is a unit of $A$. We have 
$\dis c=\frac{hp\ga^{p-1}}{\al \z^{p-1}}, c_0=\frac{h_0p\ga_0^{p-1}}{\al_0 \z^{p-1}}$ and hence,  $\dis \frac{c_0}{c}= \lb \frac{u\ga_0}{\ga}\rb^{p-1}=a^{p-1}$. We will verify that $d\al'_0=a d\al'$, the rest follows (see 6.3.3 \cite{V}).

\begin{equation} \al'_0=\ga_0(\al_0 -1) / \z=\ga a (\al u -1)/u\z =\ga a (\al-1)/\z +\ga a (u-1)/u\z=a\al'+\ga_0(u-1)/\z
\end{equation}
Since $v(\al_0-1)=v(\al u -u+u-1)<v(\al -1), v (u-1)=v(\al_0-1)$. Therefore, $\la=\ga_0(u-1)/\z \in A^{\times}$ and we have 
$d \al_0'=a d\al'+\al' da+d\la=ad\al',$ in $\Om^1_{A[\al']|A}$.
\end{proof}

\noindent Due to the defect, we have
$$\is= \js= \lb \left\{ \frac{\si(b)}{b}-1 \mid b \in B^{\times}
\right\}\rb B= \lb \{\si(b)-b \mid b \in B^{\times} \} \rb B$$

By \Cref{fil} and  $ v(\si(\al')-\al')= v \lb \lb \frac{\ga_\al}{\z}\rb (-\z \al) \rb =v(-\ga_\al \al)$, we have

\begin{equation}
 \is = \js =
\lb \{ \si(\al')-\al' \mid \al \in \mathscr{S}_0 \} \rb B= \lb \{\ga_{\al} \mid \al \in \mathscr{S}_0  \} \rb B
\end{equation}

\begin{lm}\label{diff} Consider the fractional ideals $\T$ and
$\T'$ of $L$ given by $\T= \{x \in L \mid x\ga_0 \in \js \}$ and\\
$\T'=\{ x \in L \mid x\ga_0 \in \ns B \}$. Then we have:
\begin{enumerate}[(a)]
\item $\Om^1_{B|A} \cong \T/\T'$ 
\item $\T/ \js \T \cong \js/\js^2$ 
\end{enumerate}
\end{lm}
\begin{proof}
\begin{enumerate}[(a)]
\item Let $I$ be the fractional ideal of $L$ generated by the
elements $(\frac{1}{a_\al})$. Let $I'$ be the fractional ideal
of $L$ generated by the elements $(\frac{c_\al}{a_\al})$. Under
the isomorphisms described in the preceding discussion, we can
identify each $\dis \Om^1_{A[\al']|A}$ with $
(\frac{1}{a_\al})A[\al']/(\frac{c_\al}{a_\al})A[\al']$. Taking
limit over $\al$'s, we can identify $\Om^1_{B|A}$ with $I/I'$.

\noindent Since $-v(a_\al)=v(\ga_\al)-v(\ga_0)= v(\ga_\al)-\mu$,
$\dis I = \T$.
Similarly, $v(c_\al)=(p-1)v(\ga_\al)$ implies that $
v(\frac{c_\al}{a_\al})=pv(\ga_\al)-\mu$ and hence, $ I'=\T'$.

\item This follows from the fact that $\T \cong \js $ as
$B$-modules, via the map $\times \ga_0 \al_0: x \mapsto x \ga_0 \al_0$.
\end{enumerate}
\end{proof}

\subsubsection{Refined Swan Conductor and Proof of \Cref{comm dia} in the defect case} \noindent Let $L|K$ be a defect extension as in 1.1 for the rest of this section.
\begin{defn} Consider the ideals $\I_\al$ of $A$ defined for each $\al \in \mathscr{S}_0$ by $$\I_\al:= \left\{ x \in K \mid v(x) \geq \lb \frac{p-1}{p}\rb  v \lb \frac{\z^p}{h_\al-1} \rb \right\}$$ and let $$\I:= \cup_{\al \in \mathscr{S}_0} \I_\al$$
We note that the definition of $\I_\al$ only depends on the valuation of $h_\al-1.$\\
The \tit{refined Swan conductor} $\rsw$ of the extension $L|K$ is defined to be the $A$-homomorphism\\ $\rsw: \cH \to  \omk/\I \omk$ given by $\dis r \mapsto \frac{r  }{\z^p}\dlo h_\al,$ where $r \in \lb \frac{\z^p}{h_\al-1} \rb$ for some $\al \in \mathscr{S}_0$. 

\end{defn}
\noindent We will show, as before, that this definition does not depend on the choice of $h_\al$.
\begin{lm}\label{rsw defect} \begin{enumerate}[(i)] 
\item The map $\rsw$ in this case, is well-defined.
\item For each $\al \in \mathscr{S}_0, \lb \frac{\z^p}{h_\al-1} \rb \subset \I_\al \subset \lb \frac{\z}{\al-1} \rb A[\al'] \cap A$
\item $\cH \subset \I \subset \is \cap A$
\end{enumerate} 
\end{lm}
\begin{proof} \begin{enumerate}[(i)] 
\item  Let $h_1=h_{\al_1}, h_2=h_{\al_2}$ for some $\al_1, \al_2 \in \mathscr{S}_0$ and $r \in \lb \frac{\z^p}{h_1-1} \rb \cap \lb \frac{\z^p}{h_2-1} \rb.$ It is enough to focus on the case when $v(h_1-1) \neq v(h_2-1)$. We imitate the proof of \Cref{rsw well-def}.\\ Let $h_2=a^ph_1; a  \in A^{\times}, a \neq 1$ and without loss of generality, assume that $v(h_1-1) < v(h_2-1).$ It is enough to show that $v(a^p-1) \geq v(h_1-1)$. Since
$$v(h_2-1)=v(a^ph_1-1)=v((a^p-1)h_1+(h_1-1)) > v(h_1-1),$$
we must, in fact, have $v(a^p-1) = v(h_1-1)$. Hence,  $\rsw$ is well-defined in this case.
\item The first part $\lb \frac{\z^p}{h_\al-1} \rb \subset \I_\al$ is easy to see. The next part follows from 
$$ v \lb \frac{\z}{\al-1} \rb = \lb \frac{1}{p} \rb v \lb \frac{\z^p}{h_\al-1} \rb \leq \lb \frac{p-1}{p} \rb v \lb \frac{\z^p}{h_\al-1} \rb .$$
\item This follows directly from (ii), since $\is = \js = \lb \left\{ \frac{\z}{\al-1} \mid \al \in \mathscr{S}_0 \right\}  \rb$. 
\end{enumerate}
\end{proof}

\begin{proof}[Proof of \Cref{comm dia} for the defect case] Let $L|K$ be a defect extension as in 1.1. 
\Cref{alpha level} and \Cref{diff} allow us to write $\OL=\varinjlim_{\al \in \mathscr{S}_0} \Om^1_{A[\al']|A}$ and it is
enough to consider the diagram for each $\al \in \mathscr{S}_0$:

\begin{equation}
\begin{tikzcd}[column sep=large]
\Om^1_{A[\al']|A}/\lb \frac{\z \al}{\al-1}\rb A[\al']\Om^1_{A[\al']|A}
\arrow{r}{\varphi_\si}[swap]{\cong} \arrow{d}[swap]{\dn}
&\lb \frac{\z \al}{\al-1}\rb A[\al']/\lb \frac{\z \al}{\al-1}\rb ^2 A[\al']
\arrow[hook]{d}{\overline{N}}\\
\OK /(\is \cap A )\OK & \lb \frac{\z^ph_\al}{h_\al-1}\rb A/\lb \frac{\z^ph_\al}{h_\al-1}\rb^2A
\arrow{l}{\rsw}
\end{tikzcd}\end{equation}

 where the maps are given by 

\begin{center}$\begin{tikzcd}[column sep=large]
b d  \al' \arrow{r}{\varphi_\si}[swap]{} \arrow{d}[swap]{\dn}
&b\al' \lb \frac{\z \al}{\al-1}\rb \arrow{d}{\overline{N}}\\
N(b\al') \dl (h_\al-1)  &N(b\al') \lb \frac{\z^ph_\al}{h_\al-1}\rb \arrow{l}{\rsw}
\end{tikzcd}$\end{center}

We note that in $\oml, \dl (\al-1)= \dl \al' + \dl \lb \frac{\z}{\ga_\al} \rb = \dl \al'=\frac{d \al'}{\al'}$ and
$\frac{\si(\al')}{\al'}-1=\frac{\z \al}{\al-1}$.

At each $\al$-level, we observe the following:
\begin{enumerate}[(i)]
\item The map $\varphi_\si:
\Om^1_{A[\al']|A}/(\frac{\z \al}{\al-1})\Om^1_{A[\al']|A} \to
\lb \frac{\z \al}{\al-1}\rb A[\al']/\lb \frac{\z \al}{\al-1}\rb ^2 A[\al']$ is same as the one obtained
from \Cref{diff}.  This can be proved as follows.\\ By \Cref{diff},
$\Om^1_{A[\al']|A}/(\frac{\z \al}{\al-1})\Om^1_{A[\al']|A} \cong
(\frac{1}{a_\al})/(\frac{\z \al}{\al-1})(\frac{1}{a_\al}) \cong
(\frac{\z \al}{\al-1})/(\frac{\z \al}{\al-1})^2$ under the composition $d
\al' \mapsto \frac{1}{a_\al} \mapsto \ga_0 \al_0
\frac{1}{a_\al}=\al \ga_\al$.

On the other hand, $\varphi_\si(d \al')=\al' \lb \frac{\z \al}{\al-1}
\rb = \al \ga_\al$.

\item The map $\rsw$ is well-defined. We just need to verify that for $h=h_\al$, $\rsw \lb \frac{\z^ph}{h-1}\rb = \dl (h-1)$. By definition, $$\rsw\lb  \frac{\z^p h}{h-1} \rb:=\lb \frac{h}{h-1}\rb (h-1) \lb \frac{\dl (h-1)}{h}\rb=\dl (h-1)=\dl N(\al-1).$$

\end{enumerate}

This concludes the proof.

\end{proof}

\section{Results for  the non-Kummer Case}
\noindent In the $p=2$ case, we always have $\ze=-1 \in K$. For the rest of this section, we will assume $p>2$.

\noindent 
\begin{notn} Let $K'$ be a valued field of characteristic $0$ with henselian valuation ring $A'$,  valuation $v'$ and residue field $k'$ of characteristic $p>0$. Consider  a non-trivial Galois extension $L'|K'$ of degree $p$, with Galois group $G':=\Gal(L'|K')$. Let $w', B', l'$ denote the valuation, valuation ring and the residue field of $L'$. We consider the fields $K:= K'(\ze)$, $L
:=L'(\ze)$ and  the Kummer extension $L|K$ described by $\al^p=h$ for some $h \in K$. \\ 
The Galois group $G:=\Gal(L|K)$ is cyclic of order $p$, generated by $\si: \al \mapsto \ze \al$. Let $\La_K:= \Gal(K|K')$ and $\La_L:= \Gal(L|L')$, we will omit the subscripts when the meaning is clear. Note that the order  of $\La$ is coprime to $p$. We will use the notation of 1.1 for the extension $L|K$.
\end{notn}
\subsection{Invariants for $L'|K'$}
 First we define the corresponding invariants for the extension $L'|K'$ as follows.
\begin{equation}\label{i'} \cI' := \ \lb \left\{ \si(b)-b \mid b \in B'
\right\} \rb \sub B'
\end{equation}
\begin{equation}\label{j'} \cJ' :=  \lb \left\{ \frac{\si(b)}{b}-1 \mid b \in L'^{\times}
\right\} \rb \sub B'
\end{equation}
\begin{equation}\label{n'} \cN' := \lb N_{L'|K'}(\cJ') \rb \subset A'
\end{equation} 
\begin{equation}\label{h'} \cH' := (\cH)^{\La} \subset A'
\end{equation}
We prove the following lemma in order to prove \Cref{'} and further results.
\begin{lm}\label{mu ana} Let $L|L'$ be as above, $[L:L']=m$, where $m$ is a positive integer coprime to $p$. Assume that $L|L'$ is either unramified or totally ramified. Then there exists an $L'$-basis $\{b_i\}_{1 \leq i \leq m}$ of $L$ that satisfies the following properties.
\begin{enumerate}[(B1)]
\item $\{b_i\}_{1 \leq i \leq m}$ is also a $K'$-basis of $K$.
\item $\{b_i\}_{1 \leq i \leq m} \sub A$
\item If $L|L'$ is totally ramified, the valuations $\{w(b_i)\}_{1 \leq i \leq m}$ are all distinct modulo the value group of $L'$. If $L|L'$ is unramified, the residue classes $ \{\overline{b_i}\}_{1 \leq i \leq m}$ form a basis of the residue extension $l|l'$. 
\item For any $0 \neq x=\sum_{i=1}^m x_ib_i; x_i \in L',$ we have $w(x)=\min_{i} w(x_ib_i)$.
\item For any $0\neq x=\sum_{i=1}^m x_ib_i; x_i \in L'$as above, $x \in B \iff x_i b_i \in B$ for all $1 \leq i \leq m$.
\end{enumerate}
\end{lm}
\begin{proof} \textbf{(B1-3)}: If $L|L'$ is totally ramified, the ramification indices $e_{L|L'}$ and $e_{K|K'}$ are both equal to $m$. We can choose $m$ elements $\{b_i\}_{1 \leq i \leq m}$ of $K$ that have distinct valuations modulo the value group of $K'$. Without loss of generality, we may assume that they have non-negative valuations.

If $L|L'$ is unramified, $[l:l']=m=[k:k']$ and we can choose units $\{b_i\}_{1 \leq i \leq m}$ of $K$ satisfying the required conditions. 

\textbf{(B4)}: If $L|L'$ is totally ramified, $w(x_i b_i)$ are all distinct by (B3), and therefore, exactly one term achieves the minimum valuation. 

If $L|L'$ is unramified, it is possible for more than one term to have the minimum valuation. However,  $x$ cannot have a greater valuation. This can be proved as follows.
Without loss of generality, let $w(x_1b_1)= \min_i w(x_ib_i)$. If $w(x)>w(x_1b_1)=w(x_1),$
$ \frac{x}{x_1}=b_1 + \sum_{i=2}^m \lb \frac{x_i}{x_1}\rb b_i  \in \m_L$. Since $\overline{b_i}$ are $l'$-linearly independent, this is not possible.

\textbf{(B5)}: This follows from (B4).
\end{proof}
\begin{pr}\label{'} We have the following relations between the invariants for $L|K$ and the invariants for $L'|K'$
\begin{enumerate}
\item $\js =\cJ'B$
\item $(\js)^{\La}=\cJ'$
\item $\ns = \cN'A$
\item $(\ns)^{\La}=\cN'$
\item $(\is)^{\La} = \cI'$
\end{enumerate}

\end{pr}
\begin{proof} Let $[L:L']=m=e_0f_0$, where $m$ is a positive integer coprime to $p$, $e_0$ is  the ramification of $L|L'$ and $f_0$ is the inertia degree of $L|L'$. It is enough to consider the two cases where $L|L'$ is either unramified or totally ramified. This can be seen by considering the two extensions $T|L'$ and $L|T$, where $T$ is the maximal unramified subextension of $L|L'$.  

Let  $\{b_i\}_{1 \leq i \leq m}$ be an $L'$-basis of $L$ as described in \Cref{mu ana}.

\begin{enumerate}
\item Let $x=\sum_{i=1}^m x_ib_i; x_i \in L'$ be an element of $L^{\times}$. Since $\si$ fixes each $b_i$ (by (B1)), we have
$$ \frac{\si(x)}{x} -1 = \sum_{i=1, x_i \neq 0}^m \lb \frac{\si(x_ib_i)}{x_ib_i} -1 \rb \frac{x_ib_i}{x}=\sum_{i=1, x_i \neq 0}^m \lb \frac{\si(x_i)}{x_i} -1 \rb \frac{x_ib_i}{x}$$

For each $i$ with $x_i \neq 0, \frac{\si(x_i)}{x_i} -1 \in \cJ'$ and $\frac{x_ib_i}{x} \in B$ (by (B4)).
Thus, $\js \sub \cJ'B$. The reverse direction is trivial.
\item It is clear that $(\js)^{\La}$ contains $\cJ'$. For the reverse direction, consider the action of $\tr_{L|L'}$ on $\js$. For $x$ as above, $$\tr_{L|L'}\lb \frac{\si(x)}{x} -1 \rb =  \sum_{i=1, x_i \neq 0}^m \lb \frac{\si(x_i)}{x_i} -1 \rb \tr_{L|L'}\lb \frac{x_ib_i}{x}\rb \in \cJ'$$
On the other hand, $\tr_{L|L'}$ acts on $\lb \js \rb^{\La}$ as multiplication by $m$. The rest follows from $(m, p)=1$. 
\item By 1, for any $x \in \js=\cJ'B$,  there exists $y \in \cJ'$ such that $ x \in (y)B$. Thus, $\ns=\lb N_{L|K}(\cJ'B)\rb=\lb N_{L'|K'}(\cJ')\rb A=\cN'A$
\item This follows from 2 and 3.
\item  For any $x \in B, \tr_{L|L'}\lb(\si-1)(x)\rb=(\si-1)\lb \tr_{L|L'}(x)\rb \in \cI'.$ The rest  of the proof is quite similar to the proof of 2.

\end{enumerate}

\end{proof}

\subsection{Main Results for $L'|K'$} Observe that $L'|K'$ and $L|K$ have the same defect. We have the analogues of the main results as follows.
\begin{thm}\label{hn'} $\cH'=\cN'$.
\end{thm}
\begin{proof}This is a direct consequence of \Cref{hn} and \Cref{'}.\end{proof}
\begin{thm}\label{comm dia'} By taking $\La$-invariant parts of the commutative diagram \begin{center}$
\begin{tikzcd}[column sep=large]
\oml/\js \oml \arrow{r}{\varphi_\si}[swap]{\cong}
\arrow{d}[swap]{\dn}
&\js/\js^2  \arrow[hook]{d}{\overline{N_{L|K}}}\\
\omk/(\is \cap A )\omk   &\cH/\cH^2 \arrow{l}{\rsw}
\end{tikzcd}$
\end{center}
we have the following commutative
diagram for $L'|K'$:

\begin{center}$
\begin{tikzcd}[column sep=large]
\om_{B'|A'}^1/\cJ' \om_{B'|A'}^1 \arrow{r}{\varphi'}[swap]{\cong}
\arrow{d}[swap]{\dn'}
&\cJ'/\cJ'^2  \arrow[hook]{d}{\overline{N_{L'|K'}}}\\
\om_{A'}^1/(\cI' \cap A')\om_{A'}^1   &\cH'/\cH'^2 \arrow{l}{\rsw'}
\end{tikzcd}$
\end{center}

\noindent The maps $\dn, \dn'$ are induced by the norm maps $N_{L|K}$ and $N_{L'|K'}$, while  the map $\rsw'$ is the restriction of the map $\rsw$ to $\cH'/\cH^2$.\end{thm}
\begin{proof}  Validity and properties of the map $\rsw'$ follow from the commutativity of the first diagram and properties of the map $\rsw$. The rest follows from \Cref{'}.
\end{proof}

\section{Generalizing the Results of \cite{V} to Defect Extensions of Rank $>1$}

In \cite{V}, we proved the main results under the assumption that the Artin-Schreier extension $L|K$ is defectless or has valuation of rank $1$. However, we observed the following.

\begin{itemize}
\item In the case $p=2$, the results were true regardless of the rank of the valuation. This led us to believe that the results should be true for defect extensions of higher rank, even when $p>2$.
\item Many of the key lemmas, such as \Cref{js}, were proved without using the condition on the rank. 
\item If we could prove the result $\cH=\ns$ independent of the rank, the rest would follow.
\end{itemize}

We can easily modify  the proof of \Cref{hn} presented in 4.2 to fit the Artin-Schreier case. Similarly, we can imitate the proof of \Cref{diff} and thus, the main results of \cite{V} can be generalized  to the higher rank defect case.

\vspace{15pt}
\noindent \textbf{Acknowledgments:} 
 I am very grateful to 
Professor Kazuya Kato (University of Chicago) for his invaluable advice, helpful
feedback during the writing process, and his constant support
during the project.

\medskip
\medskip
\medskip \noindent Vaidehee Thatte\\Department of Mathematics and Statistics,\\
Queen's University,\\ 48 University Ave.\\ Kingston, ON Canada, K7L 3N6\\
\href{http://mast.queensu.ca/\textasciitilde vaidehee/}{http://mast.queensu.ca/\textasciitilde vaidehee/}\\
Email: \texttt{vaidehee@mast.queensu.ca}

\vfill

\begin{thebibliography}{100}
\bibitem[OH87]{OH} Hyodo, O. \tit{Wild Ramification in the Imperfect Residue Field Case}, Advanced Studies in Pure Math. 12 (1987), 287-314.

\bibitem[KK89]{K2} Kato, K. \textit{Swan Conductors for
Characters of Degree $1$ in the Imperfect Residue Field Case},
Contemp. Math. 83 (1989), 101-132.

\bibitem[FVK06]{F} Kuhlmann, F. -V. \textit{Valuation Theoretic
Aspects of Local Uniformization}, Lecture Notes for the Summer
School on Resolution of Singularities, Trieste, Italy (2006).
\bibitem[VT16]{V} Thatte, V. \tit{Ramification Theory for Artin-Schreier Extensions of Valuation Rings}, Journal of Algebra, 456C, 355-389, (2016).
\bibitem[XZ14]{XZ}L. Xiao, I. Zhukov \tit{Ramification of Higher Local Fields, Approaches and Questions}, Algebra i Analiz, vol. 26, issue 5, 1-63 (2014).

\bibitem[S]{S} Serre, J.P. \textit{Local Fields},
Springer-Verlag New York Inc., (1979).
\end{thebibliography}
\end{document}